\documentclass[letterpaper, 10 pt, conference]{ieeeconf}  % Comment this line out
\usepackage{amsmath}
\usepackage{amssymb}
\usepackage{algorithm}
\usepackage{tikz}
\usepackage[noend]{algpseudocode}
\usepackage{graphicx}
\usepackage{pgfplots}
\usepackage[noadjust]{cite}
\usepgfplotslibrary{fillbetween} 
\usepackage[caption=false, font=footnotesize]{subfig}
\let\labelindent\relax
\usepackage[shortlabels]{enumitem}
\usepackage[colorinlistoftodos,prependcaption]{todonotes}
\IEEEoverridecommandlockouts                              % This command is only
\algrenewcommand\textproc{\textsc}

\title{Lift, Partition, and Project: Parametric Complexity Certification of Active-Set QP Methods in the Presence of  Numerical Errors}

\author{Daniel Arnstr\"om and Daniel Axehill% <-this % stops a space
  \thanks{This work was supported by the Swedish Research Council (VR) under contract number 2017-04710.}% <-this % stops a space
\thanks{D. Arnstr\"om and D. Axehill are with the Division of Automatic Control, Link\"oping University,
  Sweden 
{\tt\small daniel.\{arnstrom,axehill\}@liu.se}}%
}

\begin{document}
\renewcommand{\baselinestretch}{0.91}

\definecolor{set19c1}{HTML}{E41A1C}
\definecolor{set19c2}{HTML}{377EB8}
\definecolor{set19c3}{HTML}{4DAF4A}
\definecolor{set19c4}{HTML}{984EA3}
\definecolor{set19c5}{HTML}{FF7F00}
\definecolor{set19c6}{HTML}{FFFF33}
\definecolor{set19c7}{HTML}{A65628}
\definecolor{set19c8}{HTML}{F781BF}
\definecolor{set19c9}{HTML}{999999}

\maketitle
\thispagestyle{empty}
\pagestyle{empty}
\newtheorem{proposition}{Proposition}
\newtheorem{lemma}{Lemma}
\newtheorem{corollary}{Corollary}
\newtheorem{remark}{Remark}
\newtheorem{theorem}{Theorem}
\newtheorem{definition}{Definition}
\newtheorem{assumption}{Assumption}
\newtheorem{example}{Example}

\newcommand{\certcite}{\cite{zeilinger2011certlp,cimini2017certqp,cimini2019complexity,arnstrom2022unifying}} 

\begin{abstract}
  When Model Predictive Control (MPC) is used in real-time to control linear systems, quadratic programs (QPs) need to be solved within a limited time frame. 
  Recently, several parametric methods have been proposed that certify the number of computations active-set QP solvers require to solve these QPs. These certification methods, hence, ascertain that the optimization problem can be solved within the limited time frame. 
  A shortcoming in these methods is, however, that they do not account for numerical errors that might occur internally in the solvers, which ultimately might lead to optimistic complexity bounds if, for example, the solvers are implemented in single precision. In this paper we propose a general framework that can be incorporated in any of these certification methods to account for such numerical errors.
\end{abstract}

\section{Introduction}
In Model Predictive Control (MPC) an optimization problem must be solved at each time-step, which, when used for control of safety-critical systems operating in real-time with limited hardware, requires the employed optimization solver to be efficient and robust to numerical errors \cite{johansen2017dmpc}.  

In the context of linear MPC \cite{borrelli2017lmpc}, the optimization problems in question are quadratic programs (QPs) that depend on the system state and setpoints. A popular class of methods for solving such QPs is active-set methods \cite{ferreau2014qpoases,goldfarb,nocedal,daqp}, which, in the context of MPC, have the favourable properties of being numerically stable (compared with first-order methods) and of being straight-forward to warm start. A notorious drawback, especially unfavourable in real-time applications, is, however, that their worst-case computational complexity is exponential in the number of decision variables \cite{klee1972good}; but the actual computational complexity is often far from the worst case in practice \cite{spielman2004smoothed}. To close this gap between theory and practice, methods that determine \textit{a priori} worst-case bounds on the complexity for different active-set methods have been proposed in \certcite. These methods determine exactly which sequence of linear equation systems that must be solved (and hence the exact number of FLOPs) to compute an optimal solution, for any QP that can arise in a given linear MPC application. In particular, these QPs belong to a family of QPs that are parametrized by system states and setpoints \cite{bemporad2002explicit}. More concretely, then, the certification methods iteratively partition the parameter space into finer regions, where each such region contains states and setpoints that generate the same sequence of linear equation systems. 

A shortcoming in these complexity certification methods is, however, that none of them take into account numerical errors inside the optimization solvers to be certified. Hence, the resulting worst-case complexity certificates are only valid in perfect arithmetic, which might suffice when double precision is used. In practice, however, the solvers are often implemented on limited hardware that require single or fixed-point precision to be used to fulfill real-time constraints. The resulting reduction in precision not only reduces the quality of the solution, it can also cause the solver to cycle. By not considering such numerical errors, then, the certification methods might provide optimistic bounds on the computational complexity, and in the worst-case the methods might signal finite computational complexity when in actuality the complexity is infinite due to cycling from numerical errors.

In the context of MPC, numerical errors, in particular round-off errors, have mainly been analyzed for fixed-point implementations of first-order methods \cite{mcinerney2019roundoff,patrinos2015fixedpoint, jerez2014embedded}. In the mathematical programming literature, methods for mitigating cycling in active-set QP methods have been proposed in \cite{fletcher1993resolving,gill1989practical}, where the former ensures finite termination even in the presence of round-off errors. A drawback of such anti-cycling schemes is, however, that they lead to additional overhead in each iteration. Also, as is noted in both \cite{fletcher1993resolving} and \cite{gill1989practical}, a suitable choice of tolerances often suffices, with the caveat that such a suitable choice is problem dependent and, hence, impossible to make \textit{a priori} if a large problem class should be handled.

The main contribution of this paper, presented in Section~\ref{sec:errors}, is an extension to the complexity certification methods in \cite{zeilinger2011certlp,cimini2017certqp,cimini2019complexity,arnstrom2022unifying} that enables numerical errors to be accounted for. The extension builds on a three-step approach consisting of: (i) $\textit{lifting}$ the parameter space to include numerical errors; (ii) $\textit{partitioning}$ the parameter space based on the solver's behaviour; (iii) $\textit{projecting}$ down the new regions onto the nominal parameter space.   

The extension can be used to evaluate different anti-cycling schemes, both theoretically-grounded ones such as \cite{fletcher1993resolving,gill1989practical}, or, as is exemplified in Section \ref{sec:result}, more ad hoc schemes. Moreover, as is also exemplified in the experiments in Section \ref{sec:result}, the extension can be used to select appropriate tolerances for the specific QPs that need to be solved in a linear MPC applications, and the sufficiency of these tolerances can be ascertained a priori for these QPs. 

An additional, minor yet pivotal, contribution of the paper is the abstraction made in Section \ref{sec:par-cert} of the certification methods in \cite{zeilinger2011certlp,cimini2017certqp,cimini2019complexity,arnstrom2022unifying} in terms of parameter-dependent finite automatons, which gives a unified characterization of the methods therein.

\section{Parametric complexity certification methods}
\label{sec:par-cert}
It is well-known (see, e.g., \cite{borrelli2017lmpc}) that the optimization problems that need to be solved in each time-step in linear MPC often take the form
\begin{equation}
  \label{eq:mpqp}
  \begin{aligned}
	x^*(\theta) = &\:\:\:\underset{x}{\text{argmin}}&& \frac{1}{2}x^T H x + f(\theta)^T x\\
				  &\text{subject to} && C x \leq d(\theta),
  \end{aligned}
\end{equation}
where the decision variable $x\in \mathbb{R}^{n_x}$ is related to the control and the parameter $\theta \in \Theta_0 \subseteq \mathbb{R}^{n_{\theta}}$ is related to system states and setpoints. In particular, we assume that $\Theta_0$ is a polyhedron. The objective function is characterized by $H\in\mathbb{S}^{+}_{n_x}$ and the affine function $f:\Theta_0 \to \mathbb{R}^{n_x}$, while the feasible set is characterized by $C\in\mathbb{R}^{m\times n_x}$ and the affine function $d:\Theta_0 \to \mathbb{R}^m$. 

Different parameters $\theta\in \Theta_0$ in \eqref{eq:mpqp} defines different QPs; and since $\theta$ depends on system states and setpoints, both of which exact values at an arbitrary time step are unknown, all possible QPs given by $\theta\in \Theta_0$ might need to be solved online. 
The main goal of the certification methods in \cite{zeilinger2011certlp,cimini2017certqp,cimini2019complexity,arnstrom2022unifying} is to determine $\text{exactly}$ how different active-set algorithms ``behave'' when solving QPs corresponding to any $\theta \in \Theta_0$. Since a complete description of the methods in \certcite\:is out of scope of this paper, we give an abstract representation of both the algorithms that they certify and the certification methods themselves in Section \ref{ssec:gen} and \ref{ssec:cert-gen}, respectively. 
Stripping away implementation-specific details through this abstraction allows us to focus on the main contribution of this paper: a unified strategy, presented in Section \ref{sec:errors}, to analyze the effect of numerical errors in the solvers certified in \certcite. 

Formally proving the correspondence between this abstraction and the actual methods is also out of scope of this paper. To at least make the abstraction plausible to the reader, we use a running example of how the abstract representation maps onto a step performed in a dual active-set method.    

\subsection{Generic algorithm to be certified}
\label{ssec:gen}
By considering a specific QP, i.e., assuming that $\theta$ in \eqref{eq:mpqp} is fixed, the active-set methods certified in \cite{zeilinger2011certlp,cimini2017certqp,cimini2019complexity,arnstrom2022unifying} can be represented as finite automatons \cite{sipser1996introduction}, where the state $q\in Q$ of the automaton is called the \textit{solver state}, and $Q$ is the set of all possible solver states. The solver state at iteration $k$ is denoted $q_k$. 
For those familiar with active-set methods, the solver state relates to the \textit{working set} (see, for example, \cite[§16.5]{nocedal} for an introduction to active-set methods). 
A sequence of solver states $\{q_k\}_{k}$ produced by a solver is called the \textit{behaviour} of that solver (given the QP). We use the notation $\{e_i\}_{i=1}^N$ for a sequence of $N$ elements and, as we did for the solve-state sequence above, often drop $N$ if its cardinality is unimportant. 

In our automaton representation of the active-set algorithms, the state update takes a particular form, based on an \textit{intermediate variable} $z \in \mathbb{R}^{n_z(q)}$ and a set of polyhedra $\{z:A^i z \leq b^i\}_i$, both of which are generated based on the current solver state. For a concrete example of $z$ and $\{z:A^iz\leq b^i\}_i$, see Example \ref{ex:dual}. 
In iteration $k$ the solver state $q_k$ is updated to state $q_{k+1}$ by the transition function $\delta:Q\times \mathbb{Z} \to Q$, where the second argument is an index $i$, given by $i: A^i z_k \leq b^i$. In other words, if $z_k \in \mathcal{Z}^i \triangleq \{z:A^i z\leq b^i\}$, the performed state update is $q_{k+1} = \delta(q_k,i)$. 
This update is well-defined if the polyhedra $\{\mathcal{Z}^i\}_i$ form a partition of $\mathbb{R}^{n_z(q)}$.
\begin{assumption}
  \label{as:z-part}
  The set of polyhedra $\{\mathcal{Z}^i\}_i$ partitions $\mathbb{R}^{n_z(q)}$, i.e., $\mathring{\mathcal{Z}}^i \cap \mathring{\mathcal{Z}}^j = \emptyset$ and $\cup_i \mathcal{Z}^i = \mathbb{R}^{n_z(q)}$.  
\end{assumption}

Assumption \ref{as:z-part} ensures that an index $i: A^i z_k \leq b^i$ exists and is unique for any $z_k$ in the interior of a region. 

The above-mentioned steps are summarized in Algorithm~\ref{alg:gen}.
\begin{algorithm}[H]
  \caption{Generic formulation of the algorithms certified in \cite{zeilinger2011certlp,cimini2017certqp,cimini2019complexity,arnstrom2022unifying} as a finite automaton.}
  \label{alg:gen}
  \begin{algorithmic}[1]
	\Require $q_0$, $\mathbb{Q} \leftarrow \emptyset$
	\Ensure $\mathbb{Q}$
	\State $k\leftarrow 0$
	\Repeat
	\State $\mathbb{Q} \leftarrow \mathbb{Q}\cup\{q_k\}$ 
	\State Generate $\{A^i z \leq b^i\}_i$ and $z_k$ based on $q_k$ 
	\State  $i \leftarrow$ find $i$ such that $A^i z_k\leq b^i$ \label{step:i-loc} 
	\State $q_{k+1} \leftarrow$ $\delta(q_k,i)$ \label{step:state-up} 
	\State $k\leftarrow k+1$
	\Until{$q_k$ marks termination} 
	\State \textbf{return} $\mathbb{Q}$.
  \end{algorithmic}
\end{algorithm}
\begin{example}[Dual active-set algorithms]
 \label{ex:dual}
 To be more concrete, we briefly relate how an iteration of Algorithm \ref{alg:gen} maps onto an iteration of a dual active-set algorithm \cite{daqp}. In particular, we consider, for simplicity, iterations in which primal feasibility is investigated by evaluating whether the primal slack $\mu$ is nonnegative. 

In dual active-set algorithms, if $\mu\geq -\epsilon_p$ (where $\epsilon_p>0$ is a user-specified tolerance) a global solution has been found and the algorithm terminates. Otherwise, if $\mu\ngeq -\epsilon_p$, the most negative component of $\mu$, that is,  ${i = \text{argmin}_{j:\mu<-\epsilon_p} \mu_j}$, is used to update the solver state; concretely, updating the solver state $i$ here means adding $i$ to the working set. The primal slacks $\mu$ that result in the $i$th component being the most negative can be expressed as the polyhedron 
\begin{equation}
  \label{eq:mu-add}
\mathcal{P}^i = \{\mu:\mu_i< -\epsilon_p,\:\: \mu_i \leq \mu_j, i\neq j \},
\end{equation}
and values of $\mu$ that lead to termination can be expressed as the polyhedron 
\begin{equation}
  \label{eq:mu-opt}
  \mathcal{P}^* = \{\mu: \mu \geq -\epsilon_p\}.
\end{equation}

Such an iteration of a dual active-set method, hence, maps onto an iteration of Algorithm~\ref{alg:gen} by letting $z\equiv\mu$ and the polyhedra $\{z:A^i z \leq b^i\}_i$ be the polyhedra defined by \eqref{eq:mu-add} and \eqref{eq:mu-opt}.
\end{example}

\begin{remark}[Other solver modes]
  Example \ref{ex:dual} considers an iteration performed in a particular ``mode'' of the considered active-set methods (cf. Section IV.A in \cite{arnstrom2022unifying}). Iterations in other modes do, however, also take the form of an iteration in Algorithm \ref{alg:gen}. 
\end{remark}
\begin{remark}[Clarifying the output of Algorithm \ref{alg:gen}]
  Since our interest herein is the solver's behaviour rather than the solution it produces, Algorithm \ref{alg:gen} outputs the sequence of solver states $\mathbb{Q}$, while the certified algorithms in practice output the solution to an optimization problem. 
  This solution is, however, completely determined by the final solver state and is, hence, completely determined by  $\mathbb{Q}$.  

\end{remark}
\subsection{Parametric simulation}
\label{ssec:cert-gen}
As mentioned above, representing the considered active-set algorithms by Algorithm \ref{alg:gen} is valid when a specific QP is given, i.e., it assumes that $\theta$ in \eqref{eq:mpqp} is fixed.      
Now, consider instead the entire parametric family of QPs in \eqref{eq:mpqp} parametrized by ${\theta\in\Theta_0}$. A parameter-dependent problem makes the intermediate variable $z$ depend on $\theta$, that is, $z: \Theta \to \mathbb{R}^{n_z(q)}$. 

The parameter dependence of $z$ implies, in turn, that the index $i$ in Step \ref{step:i-loc} becomes parameter dependent. 
Hence, parameters in a parameter region $\Theta$ that yield $i$ in Step \ref{step:i-loc} 
are given by the region 

\begin{equation}
  \label{eq:THi}
  \Theta^i \triangleq \{\theta\in\Theta :A^i z(\theta) \leq b^i\}.
\end{equation} 

Since the state update in Step \ref{step:state-up} is completely determined by $i$ and the current state $q$, each region in \eqref{eq:THi} corresponds to a different update of the solver state.   
The main idea behind the certification methods in \cite{zeilinger2011certlp,cimini2017certqp,cimini2019complexity,arnstrom2022unifying} is to iteratively partition the parameter space into regions of the form \eqref{eq:THi}. Concretely, a parameter region $\Theta$ is partitioned into $\{\Theta^i\}_i$ each time an iteration of Algorithm \ref{alg:gen} is performed. Partitioning the parameter space into finer and finer regions can, hence, be interpreted as simulating the algorithm parametrically, where all parameters in a region signify that they generate the same sequence of solver states, i.e., the same behaviour.

%\begin{remark}
%  Note that direct parameter dependence only enters in the intermediate variable $z$ and not in the polyhedra $\{z: A^i z \leq b^i\}_i$ and in the transition function $\delta$. 
%\end{remark}

What makes the partitioning performed in \certcite, and hence the certification methods themselves, tractable is the following structure of $z$:
\begin{assumption}[Affine intermediate variable]
  \label{as:affine}
The intermediate variable $z:\Theta \to \mathbb{R}^{n_{z}(q)}$ is an affine function, i.e., $z(\theta) = F_q \theta + g_q$ for some $F_q\in \mathbb{R}^{n_z(q) \times n_{\theta}}$, $g_q\in \mathbb{R}^{n_z(q)}$.
\end{assumption}

\begin{remark}[Affine primal slack]
  To relate back to Example \ref{ex:dual}, it is well-known that the primal slack $\mu$ (which we earlier related to $z$) of KKT-points to \eqref{eq:mpqp} is affine in $\theta$ (see, e.g., \cite[§4]{bemporad2002explicit}).
\end{remark}

By imposing Assumption \ref{as:affine} on $z$, the iterative partitioning of the parameter space described above is done with half-planes: 
\begin{lemma}[Polyhedral partition in parameter space]
  Let Assumption \ref{as:z-part} and \ref{as:affine} hold, i.e., that $\{z:A^i z \leq b^i\}_i$ partition $\mathbb{R}^{n_z(q)}$ and $z$ is an affine function of $\theta$, and assume that $\Theta$ is a polyhedron; then the regions $\{\Theta^i\}_i$ given by \eqref{eq:THi} form a polyhedral partition of $\Theta$. 
\end{lemma}
\begin{proof}
  Inserting $z(\theta) = F_q \theta + g_q$ into \eqref{eq:THi} results in $\Theta^i = \{\theta\in \Theta: A^i F_q \theta \leq b^i-A^i g_q\}$, which is an intersection of two polyhedra, i.e., a polyhedron.        
  That these polyhedra form a partition follows directly from Assumption \ref{as:z-part} and that $z(\theta)=F_q \theta + g_q$ is single-valued.  
\end{proof}

A one-dimensional example of the partitioning of a parameter region due to an iteration of Algorithm \ref{alg:gen} when $z(\theta)$ is affine in $\theta$ is visualized in Figure \ref{fig:normal_part}. 

\begin{figure}[h]
  \centering
  \begin{tikzpicture}[scale=0.9, transform shape]
  
  \coordinate (base1) at (1.2,0);
  \coordinate (base2) at (2.2,0);
  \draw[dashed,->] (0,0) -- (3,0) node[right]{$\theta$};
  \draw[ultra thick, violet!50] (0.25,0) -- node[above]{\textcolor{violet!40!black}{$\Theta$}} (2.75,0);  
  
  \draw[thick,->] (2.25,-0.2) to [out=-30,in=-150] node[below]{Partition} (4.75,-0.2);
  
  \begin{scope}[shift={(4,0)}]
  \draw[dashed,->] (0,0) -- (3,0) node[right]{$\theta$};
  \draw[ultra thick, blue!50] (0.25,0) -- node[above]{\textcolor{blue!80!black}{${\Theta}^1$}} (1.2,0);  
  \draw[ultra thick, red!50] (1.2,0.0) -- node[above]{\textcolor{red!80!black}{${\Theta}^2$}} (2.2,0.0);  
  \draw[ultra thick, green!50] (2.2,0) -- node[above]{\textcolor{green!80!black}{${\Theta}^3$}} (2.75,0);  
  \end{scope}
\end{tikzpicture}
  \caption{ One-dimensional illustration of partitioning step in the complexity certification methods in \cite{zeilinger2011certlp,cimini2017certqp,cimini2019complexity,arnstrom2022unifying} when considering exact arithmetic. Each new region corresponds to a different update to the solver state}%
  \label{fig:normal_part}
\end{figure}
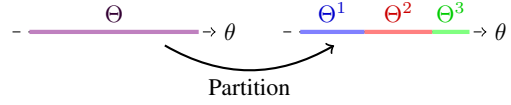

Compiling the above-mentioned ideas into an algorithm that parametrically analyzes solver-state sequences results in Algorithm \ref{alg:cert-gen}, which takes in a parameter region of interest $\Theta_0$ and a starting state $q_0$, and outputs a partition $\{\Theta^j\}_j$ and the corresponding solver-state sequences $\{\mathbb{Q}^j\}_j$. 
The algorithm maintains two stacks, $\mathcal{F}$ and $S$, which contain regions that have been terminated and regions that still need to be processed. By ``processing'' a region we mean performing a parametric iteration of Algorithm \ref{alg:gen}, resulting in the partitioning into regions defined by \eqref{eq:THi}. 

After all new regions $\Theta^i$ have been formed in an iteration, a linear feasibility problem is solved for each region to determine whether  $\Theta^i \neq \emptyset$. All nonempty sets will then be further partitioned (along the same lines as described above), unless the solver state $q$ marks termination, in which case the region and its corresponding solver-state sequence are added to the final partition $\mathcal{F}$.

\begin{algorithm}[H]
  \caption{Generic parametric simulation of Algorithm \ref{alg:gen}, which, specifically, abstracts the methods in  \cite{zeilinger2011certlp, cimini2017certqp, cimini2019complexity,arnstrom2022unifying}.}
  \label{alg:cert-gen}
  \begin{algorithmic}[1]
	\Require $\Theta_0$, $q_0$
	\Ensure $\{(\Theta^j, \mathbb{Q}^j)\}_j$
	\State Push $(\Theta_0, q_0, \emptyset)$ to $S$; $\mathcal{F} \leftarrow \emptyset$
	\While{$S\neq \emptyset$}
	\State $(\Theta,q,\mathbb{Q}) \leftarrow$ pop from $S$ 
	\State $\mathbb{Q} \leftarrow \mathbb{Q}\cup\{q\}$ 
	\State Generate $\{z:A^i z \leq b^i\}_i$ and $z(\theta)$ based on $q$ 
	\State $\{\Theta^j\}_j \leftarrow$ form $\{\theta \in \Theta: A^j z(\theta)\leq b^j \}_j$ \label{step:generate-Thi}
	\For{$\Theta^i\in\{\Theta^j\}_j : \Theta^i \neq \emptyset$}
	\State $q^i \leftarrow \delta(q,i)$  
	\If{$q^i$ marks termination} 
	\State Push $(\Theta^i, \mathbb{Q})$ to $\mathcal{F}$
	\Else
	\State Push $(\Theta^i, q^i, \mathbb{Q})$ to $S$
	\EndIf
	\EndFor
	\EndWhile
	\State \textbf{return} $\mathcal{F}$
  \end{algorithmic}
\end{algorithm}

The usefulness of Algorithm \ref{alg:cert-gen} is that it determines the behaviour of the solver, for \textit{any} parameter $\theta \in \Theta_0$, formalized in the following theorem.  
\begin{theorem}[Correctness]
  Consider a fixed $\tilde{\theta}\in \Theta_0$, and assume that the intermediate variables $z_k = z_k(\tilde{\theta})$ in Algorithm \ref{alg:gen} generate the solver-state sequence $\mathbb{Q}=\{\tilde{q}_k\}_k$. Then there exists a tuple $(\Theta^i,\{q^i_k\}_k)$ in the final partition $\mathcal{F}$ of Algorithm \ref{alg:cert-gen} such that $\tilde{\theta}\in \Theta^i$ and $\tilde{q}_k = q^i_k$, $\forall k$.
\end{theorem}
\begin{proof}
  Since this is a special case of Theorem \ref{th:corr-error} below (specifically when $\mathcal{E}=\{0\}$), we refer the reader to the proof of Theorem \ref{th:corr-error}. 
\end{proof}

As is described in \cite{arnstrom2022unifying}, the sequence of solver states $\mathbb{Q}$ for active-set methods determine \textit{exactly} which sequence of systems of linear equations need to be solved, which can, given specific implementation details of the solver, be mapped to the exact number of floating-point operations.
Hence, since Algorithm \ref{alg:cert-gen} provides the sequence of solver states $\mathbb{Q}$ for any parameter in $\Theta_0$, it can determine the exact number of flops that the active-set solver requires for any parameter in $\Theta_0$.

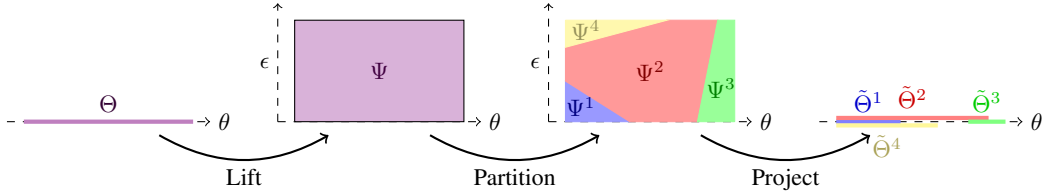
\begin{figure*}[h]
  \centering
  \begin{tikzpicture}[scale=0.9, transform shape]
  
  \coordinate (base1) at (1.2,0);
  \coordinate (base2) at (2.2,0);
  \draw[dashed,->] (0,0) -- (3,0) node[right]{$\theta$};
  \draw[ultra thick, violet!50] (0.25,0) -- node[above]{\textcolor{violet!40!black}{$\Theta$}} (2.75,0);  
  
  \draw[thick,->] (2.25,-0.2) to [out=-30,in=-150] node[below]{Lift} (4.75,-0.2);
  \begin{scope}[shift={(4,0)}]
  \draw[dashed,->] (0,0) -- (3,0) node[right]{$\theta$};  
  \draw[dashed,->] (0,0) -- node[left]{$\epsilon$} (0,1.75);  
	\draw[fill=violet!30] (0.25,0) rectangle (2.75,1.5);  
  \draw[thick,->] (2.25,-0.2) to [out=-30,in=-150] node[below]{Partition} (4.75,-0.2);
  \node[violet!50!black] (psi) at (1.5,0.75) {$\Psi$};
  \end{scope}

  \begin{scope}[shift={(8,0)}]
  \coordinate (base1) at (1.2,0);
  \coordinate (base2) at (2.2,0);
  \coordinate (v1) at (0.25,0.6);
  \coordinate (v2) at (0.25,1.1);
  \coordinate (u1) at (1.75,1.5);
  \coordinate (u2) at (2.5,1.5);

  \draw[dashed,->] (0,0) -- (3,0) node[right]{$\theta$};
  \draw[dashed,->] (0,0) -- node[left]{$\epsilon$} (0,1.75);  
  \draw[fill, blue!40] (0.25,0) -- (base1) -- (v1) --(0.25,0);  
  \draw[fill,red!40] (base1) -- (v1) -- (v2)--(u1)--(u2)--(base2)--(base1);  
  \draw[fill,green!40] (2.75,0) -- (base2) -- (u2)--(2.75,1.5)--(2.75,0);  
  \draw[fill,yellow!40] (0.25,1.5) -- (v2) -- (u1) --(0.25,1.5);

  \node[blue!50!black] (psi) at (0.475,0.225) {$\Psi^1$};
  \node[red!50!black] (psi) at (1.5,0.75) {$\Psi^2$};
  \node[green!50!black] (psi) at (2.55,0.5) {$\Psi^3$};
  \node[yellow!50!black] (psi) at (0.55,1.35) {$\Psi^4$};
  
  \draw[thick,->] (2.25,-0.2) to [out=-30,in=-150] node[below]{Project} (4.75,-0.2);
\end{scope}

  \begin{scope}[shift={(12,0)}]
  \draw[dashed,->] (0,0) -- (3,0) node[right]{$\theta$};
  \draw[ultra thick, blue!50] (0.25,0) -- node[above]{\textcolor{blue!80!black}{$\tilde{\Theta}^1$}} (1.2,0);  
  \draw[ultra thick, red!50] (0.25,0.055) -- node[above]{\textcolor{red!80!black}{$\tilde{\Theta}^2$}} (2.5,0.055);  
  \draw[ultra thick, green!50] (2.2,0) -- node[above]{\textcolor{green!80!black}{$\tilde{\Theta}^3$}} (2.75,0);  
  
  \draw[ultra thick, yellow!50] (0.25,-0.055) -- node[below]{\textcolor{yellow!60!black}{$\tilde{\Theta}^4$}} (1.75,-0.055);  
  \end{scope}
\end{tikzpicture}
  \caption{Conceptual illustration of the proposed lift-partition-project scheme to analyze intermediate errors.}
  \label{fig:lpp}
\end{figure*}
\section{Analyzing numerical errors}
\label{sec:errors}

Now, consider the case when there are numerical errors in the intermediate variable $z$, originating from, for example, round-off errors. That is, instead of $z(\theta)$ we consider $\tilde z(\theta,\epsilon) = z(\theta) + \epsilon$, where $\epsilon \in \mathcal{E}$ models the error.

To analyze how such errors affect the partitioning at Step \ref{step:generate-Thi} in Algorithm \ref{alg:cert-gen}, we propose a three-step approach: lift, partition, and project. 

First we \textit{lift} the polyhedron $\Theta$ to $\Psi = \Theta \times \mathcal{E}$, resulting in regions of the form
\begin{equation}
  \label{eq:psii}
  \Psi^i \triangleq \{\theta \in \Theta, \epsilon \in \mathcal{E} : A^i(z(\theta)+\epsilon) \leq  b^i \}.
\end{equation}
Forming the regions in \eqref{eq:psii} comprises the \textit{partition} step.
Trivially, yet importantly, we have that the nominal $\Theta^i$ can be recovered from $\Psi^i$ by fixing $\epsilon=0$: 
\begin{proposition}
  \label{prop:slice}
  For the region $\Theta^i$ defined in \eqref{eq:THi}, ${\Theta^i = \{\theta\in \Theta : (\theta,0)\in\Psi^i\}}$, where $\Psi^i$ is defined in \eqref{eq:psii}.
\end{proposition}

Lifting the polyhedron every time we want to analyze numerical errors would repeatedly increase the dimension of the extended parameter space, which quickly becomes numerically intractable. Therefore, we include a third step, a \textit{projection} step, which reduces the extended parameter space back to $\mathbb{R}^{n_{\theta}}$. 
Explicitly, projecting $\Psi^i$ onto $\Theta$ results in the region 
\begin{equation}
  \label{eq:thtilde}
  \tilde{\Theta}^i = \{\theta \in \Theta : \exists \epsilon\in \mathcal{E},\:\: (\theta,\epsilon) \in \Psi^i\}. 
\end{equation}
When $\mathcal{E}$ is a polyhedron, this projection can be carried out in practice using, for example, Fourier-Motzkin elimination \cite{dantzig1973fouriermotzkin}. In particular, we show in Section \ref{sssec:simplified-error} that the regions in \eqref{eq:thtilde} can be expressed in closed form when $\mathcal{E}$ is a hypercube.

The lift-partition-project scheme is summarized in Algorithm \ref{alg:lpp} and illustrated in Figure \ref{fig:lpp}.

\begin{algorithm}[H]
  \caption{Lift-partition-project scheme to extend the certification methods in \cite{zeilinger2011certlp, cimini2017certqp, cimini2019complexity, arnstrom2022unifying} to be able to analyze effects of numerical errors.}
  \label{alg:lpp}
  \begin{algorithmic}[1]
	\Require Region $\Theta$, half-planes $\{z:A^i z\leq b^i\}_{i=1}^N$, $ z(\theta)$, $\mathcal{E}$
	\Ensure $\{\tilde{\Theta}^i\}^N_{i=1}$
	\State Lift $\Theta$ to $\Psi = \Theta \times \mathcal{E}$. 
	\For {$i \in \{1,\dots, N\}$}  
	\State $\Psi^i \leftarrow$ partition $\Psi$ according \eqref{eq:psii}. 
	\State $\tilde{\Theta}^i \leftarrow$ project $\Psi^i$ onto $\Theta$ according to \eqref{eq:thtilde}.  
	\EndFor
  \State \textbf{return} $\{\tilde{\Theta}^i\}^N_{i=1}$
  \end{algorithmic}
\end{algorithm}

Replacing Step \ref{step:generate-Thi} in Algorithm \ref{alg:cert-gen} with the lift-partition-project scheme in Algorithm~\ref{alg:lpp} allows for absolute numerical errors to be correctly accounted for, as is formalized in the following theorem, which is the main result of this paper. 

\begin{theorem}[Correctness under numerical errors]
  \label{th:corr-error}
  Consider a given starting state $q_0$ and a sequence of errors $\{\epsilon_k\}_k$, with $\epsilon_k \in \mathcal{E}_k$. Moreover, assume that the intermediate variables $z_k=z_k(\tilde{\theta})+\epsilon_k$ in Algorithm \ref{alg:gen} generate the solver-state sequence $\mathbb{Q}=\{\tilde{q}_k\}_k$, where $\tilde{\theta}\in \Theta_0$.
  Then, if Algorithm \ref{alg:cert-gen} is started with $q_0$ and $\Theta_0$, and if Step \ref{step:generate-Thi} is replaced by the lift-partition-project scheme in Algorithm \ref{alg:lpp} (with $\mathcal{E}=\mathcal{E}_k$), 
  there exists a tuple $(\Theta^i,\{q^i_k\}_k)$ in the final partition $\mathcal{F}$, generated by Algorithm \ref{alg:cert-gen}, such that $\tilde{\theta}\in \Theta^i$ and $\tilde{q}_k = q^i_k$, $\forall k$.
\end{theorem}
\begin{proof}
  Induction step: 
  Assume that there exists a tuple $(\Theta,q_k)$ in $S$ at iteration $k$ of Algorithm \ref{alg:cert-gen} such that $\tilde{\theta}\in \Theta$ and $q_k = \tilde{q}_k$. Moreover, let $i$ be the index determined in Step \ref{step:i-loc} of Algorithm \ref{alg:gen}, i.e., $A^i({z}_k(\tilde{\theta}) + \epsilon_k) \leq b^i$.  Then, since $\mathcal{E}=\mathcal{E}_k$ is assumed to be used to generate $\tilde{\Theta}^i$, we have $(\tilde{\theta},\epsilon_k) \in \Psi^i$, which in turn implies that $\tilde{\theta} \in \tilde{\Theta}^i$. Parameters in this region result in the state update 
  \begin{equation*}
  q_{k+1} = \delta(q_k,i) = \delta(\tilde{q}_k,i) = \tilde{q}_{k+1}. 
  \end{equation*}
  Hence, the tuple $(\tilde{\Theta}^i, q_{k+1})$ such that $\tilde{\theta} \in \tilde{\Theta}^i$ and $q_{k+1} = \tilde{q}_{k+1}$ is added to the stack $S$. 
  \newline Base case: In the start of Algorithm \ref{alg:cert-gen} the tuple $(\Theta_0, q_0)$ is added to $S$, where from the premise we have that $\tilde{\theta}\in\Theta_0$ and $q_0 \triangleq \tilde{q}_0$. 
\end{proof}

Another important property when Step \ref{step:generate-Thi} in Algorithm \ref{alg:cert-gen} is replaced with the lift-partition-project scheme in Algorithm~\ref{alg:lpp} is that, if the error model $\mathcal{E}$ is correct, no redundant regions are spawned, in the following sense: 

\begin{theorem}[Nonredundancy of final partition]
  \label{th:nonred}
  Assume that Algorithm \ref{alg:cert-gen} is started with $q_0$ and $\Theta_0$, and that Step \ref{step:generate-Thi} is replaced by the lift-partition-project scheme in Algorithm \ref{alg:lpp} with $\mathcal{E}=\mathcal{E}_k$ at iteration $k$. Moreover, let $(\Theta, \{q_k\}_k)$ be any tuple in $\mathcal{F}$. 
  Then, for any $\tilde{\theta} \in \Theta$ there exists a sequence of errors $\{\epsilon_k\}_k$, $\epsilon_k \in \mathcal{E}_k$, that make Algorithm \ref{alg:gen} generate the solver-state sequence  $\{\tilde{q}_k\}_k$ which satisfies $\tilde{q}_k=q_k$, $\forall k$, if the intermediate variable at iteration $k$ of Algorithm \ref{alg:gen} is $z_k = z_k(\theta)+\epsilon_k$ and the starting state is $q_0$. 
\end{theorem}
\begin{proof}
  Induction step: assume that $\tilde{q}_k = q_k$ at iteration $k$ in Algorithm \ref{alg:gen} and 
  let $\mathcal{E}_k^i \triangleq \{\epsilon \in \mathcal{E}_k : (\tilde{\theta},\epsilon) \in \Psi^i \}$. By construction of Algorithm \ref{alg:cert-gen}, there exists a $j$ such that $q_{k+1}=\delta(q_k,j)$ and such that $\tilde{\theta} \in \tilde{\Theta}^j$, where the latter implies that $\mathcal{E}^j_k \neq \emptyset$ from \eqref{eq:thtilde}.
  Therefore selecting $\epsilon_k \in \mathcal{E}^j_k \subseteq \mathcal{E}_k$ results in $A^j(z_k(\tilde{\theta})+\epsilon_k) \leq b^j$ and the resulting state update in Algorithm \ref{alg:gen} is subsequently ${\tilde{q}_{k+1} = \delta(\tilde{q}_k,j) = \delta(q_k,j) = q_{k+1}}$. \newline 
  Base case: from the premise we have that $\tilde{q}_0 = q_0$. 
\end{proof}

\renewcommand{\baselinestretch}{0.96}
\begin{remark}[Accounting for compounding errors]
  By performing a projection step we do not investigate compounded effects of numerical errors explicitly, in the sense that we do not investigate exactly which solver-state changes an error of $\epsilon_1$ followed by an error of $\epsilon_2$ generate. Instead, compound effects are accounted for when selecting $\mathcal{E}$ (i.e., selecting $\mathcal{E}$ larger as errors compound). By making sure that $\epsilon_1 \in \mathcal{E}_1$ and that an upper-bound of the compound effect of $\epsilon_1$ and $\epsilon_2$ are contained in $\mathcal{E}_2$ in the next iteration, we know from Theorem~\ref{th:corr-error} that a region corresponding to the solver-state sequence generated by $\epsilon_1$ followed by $\epsilon_2$ will be analyzed. Because of the projection we do not, however, know the particular region. Resetting properties of some active-set methods (see, e.g., Section III.B in \cite{daqp}), ensure that $\mathcal{E}$ does not have to be increased indefinitely to account for compounding errors.
\end{remark}

\begin{remark}[Errors not captured in $\tilde{z}$]
    Another source of numerical errors in active-set methods is low-rank updates that are performed to matrix factorizations. Since the particular matrix factorizations differ for the active-set solvers that are cover by the certification methods in \certcite, we will not go into detail about those errors. Instead we note that these low-rank updates almost always only depend on $H$ and $C$ from the mpQP in \eqref{eq:mpqp}, and the solver states $\{q_k\}_k$ up until the current iteration; that is, the low-rank updates are independent of the parameter $\theta$, which allows for \textit{exactly} the same low-rank updates that are used online to be used in the certification method. Hence, the same errors will be present during the certification and will, therefore, be directly accounted for in the complexity certificates. 
\end{remark}

\subsection{Partition properties}
Generally the region $\tilde{\Theta}^i$ in \eqref{eq:thtilde} produced after the lift-partition-project step in Algorithm \ref{alg:lpp} is larger than the corresponding region $\Theta^i$ in \eqref{eq:THi}, which we formalize in the following lemma. 
\begin{lemma}[Partition properties]
  \label{lem:part-prop}
  Assume that $\mathcal{E}$ contains the origin, i.e., $0 \in \mathcal{E}$; then the following relationship between $\Theta^i$ in \eqref{eq:THi} and $\tilde{\Theta}^i$ in \eqref{eq:thtilde} hold 
  \begin{enumerate}[(i)]
    \item $\tilde{\Theta}^i \supseteq \Theta^i$,
	\item $\cup_{i} \tilde{\Theta}^i = \Theta$. 
  \end{enumerate}
\end{lemma}

\begin{proof}
  (i): If $\mathcal{E}$ only contains the origin, i.e., $\mathcal{E} = \{0\}$, Proposition~\ref{prop:slice} implies that $\tilde{\Theta}^i= \Theta^i$. If $\mathcal{E}$ contains more elements than the origin, there might $\exists \tilde{\epsilon}^i \in \mathcal{E}$ such that $(\tilde{\theta}^i,\tilde{\epsilon}^i) \in \Psi^i$ and such that $(\tilde{\theta}^i, 0) \notin \Psi^i$. 
  (For an example, compare $\Theta^2$ in Figure \ref{fig:normal_part} with  $\tilde{\Theta}^2$ in Figure \ref{fig:lpp}.)

  (ii): From the definition of $\tilde{\Theta}^i$ in \eqref{eq:thtilde} we have that $\tilde{\Theta}^i \subseteq \Theta$, which implies that $\cup_i \tilde{\Theta}^i \subseteq \Theta$. Next we have from (i) that 
  \begin{equation}
	\cup_i \tilde{\Theta}^i  \supseteq  \cup_i \Theta^i = \Theta,
  \end{equation}
since $\{\Theta^i\}_i$ forms a partition of $\Theta$.
  Taken together $\Theta \subseteq \cup_i \tilde{\Theta}^i \subseteq \Theta$, i.e., $\cup_i \tilde{\Theta} = \Theta$.
\end{proof}

\begin{corollary}
  \label{cor:part-prop}
  Under the same assumptions as in Lemma~\ref{lem:part-prop} the following relationships hold
  \begin{enumerate}[(i)]
	\item $\Theta^i = \emptyset \nRightarrow \tilde{\Theta}^i = \emptyset$ ,
	\item $\Theta^i \cap \Theta^j = \emptyset \nRightarrow \tilde{\Theta}^i \cap \tilde{\Theta}^j  = \emptyset$,
  \end{enumerate} 
\end{corollary}

Corollary \ref{cor:part-prop}, specifically (i), implies that some empty regions that are pruned when considering exact arithmetic might not be empty when considering numerical errors, and, hence, require further investigation. This implies, in turn, that some solver-state sequences that do not arise in the exact analysis might arise when the additional freedom of $\epsilon$ is considered. An example of this can be seen in Figure \ref{fig:lpp} where $\tilde{\Theta}^4\neq \emptyset$, while from Figure \ref{fig:normal_part} we have that $\Theta^4=\emptyset$.

Moreover, note that Corollary \ref{cor:part-prop}, specifically (ii), implies that there might be some overlap among the regions in $\{\tilde{\Theta}^i\}_i$, i.e., $\{\tilde{\Theta}^i\}_i$ is not a partition (but a cover) of $\Theta$. Such overlaps can also be seen in Figure \ref{fig:lpp}. The intuition behind why overlaps arise is that the \textit{same} parameter $\theta\in \Theta$ might lead to different solver-state changes for \textit{different} values of $\epsilon\in \mathcal{E}$. 

While $\{\tilde{\Theta}^i\}_i$ is not a partition of $\Theta$, it is, from Lemma \ref{lem:part-prop}, specifically (ii), a cover of $\Theta$. That is, no additional parameters $\theta\notin \Theta$ are spawned by Algorithm \ref{alg:lpp}, and all $\theta \in \Theta$ are still contained in the partition after the partitioning and projection (i.e, no holes in $\Theta$ are generated).   
%Note that $\text{dim}(\mathcal{E}) = |\mathcal{W}|$, which varies between iterations.

\subsection{Modeling the error}
The correctness of the main results in Theorem \ref{th:corr-error} and \ref{th:nonred} rely on any possible absolute error being contained in the error model $\mathcal{E}$, i.e., if we, for the set of all true absolute errors $\mathcal{E}^*$, have that $\mathcal{E}^* \subseteq \mathcal{E}$. 
Likewise, the smaller $\mathcal{E}\setminus \mathcal{E}^*$ is, the less conservative the analysis becomes. 
In this paper we are not interested in how to determine $\mathcal{E}$. Such models require insight into the particular active-set solver considered, after which standard methods in numerical analysis can be applied (see, e.g., \cite{wilkinson1960error,moore1966interval}).
Our focus is instead on how a \textit{given} model $\mathcal{E}$ can be incorporated in the certification methods in \certcite. 

Some aspects of $\mathcal{E}$ related to the parametric setting, which is non-standard in the numerical analysis literature, is briefly mentioned below. Namely, how relative errors can be transformed into absolute errors in the parametric setting, and how the projection in \eqref{eq:THi} simplifies if $\mathcal{E}$ is a hypercube. 
\subsubsection{Relative errors}
If, instead of absolute errors, we have relative errors $\epsilon_r$ such that $\tilde{z}(\theta) = (I+\text{diag}(\epsilon_r)) z(\theta)$, we can bound the corresponding absolute error $\epsilon$ as $\|\epsilon\|_{\infty} \leq \bar{\epsilon}$ with
\begin{equation}
  \label{eq:rel2abs}
  \bar{\epsilon} = \max_i \max_{\theta\in\Theta,\: \epsilon_r \in \mathcal{E}_r} [ \text{diag}(\epsilon_r) z(\theta)]_i,
\end{equation}
where $[\cdot]_i$ denotes the $i$th component of a vector.
When $z(\theta)$ is affine, $\Theta$ is a polyhedron and $\mathcal{E}_r$ is a box, the optimization problems in \eqref{eq:rel2abs} can be recast as linear programs (LP).  

\subsubsection{Simplified error model $\mathcal{E}$}
\label{sssec:simplified-error}
If $\mathcal{E}$ has additional structure, namely, is a hypercube centered at the origin
\begin{equation}
  \label{eq:Ebox}
  \mathcal{E}=\{\epsilon: \|\epsilon\|_{\infty} \leq \bar{\epsilon} \},
\end{equation}
the projected region $\tilde{\Theta}^i$ in \eqref{eq:THi} takes the closed form 
\begin{equation}
  \label{eq:th-box}
  \tilde{\Theta}^i = \{\theta \in \Theta: A^i z(\theta) \leq b^i + \|A^i\|_1 \bar{\epsilon} \}, 
\end{equation}
where we define $\| \cdot \|_1$ of a matrix as the 1-norm evaluated row-wise. 

The case when the sides are not of equal lengths, or if the center is not the origin, is directly handled by a translation followed by a scaling, which retains the polyhedral structure. 

\section{Numerical Experiments}
\label{sec:result}
To illustrate how the proposed lift-partition-project scheme can be used to analyze the behaviour of a solver in the presence of numerical errors, we consider the dual active-set algorithm in \cite{daqp}, which is covered by the complexity certification framework in \cite{arnstrom2022unifying}. As in Example \ref{ex:dual}, we consider absolute errors in the primal slack $\mu$ (computed at Step 5 in Algorithm 1 in \cite{daqp}). For simplicity, we use the error model $\mathcal{E}=\{\epsilon:\|\epsilon\|_{\infty} \leq \bar{\epsilon}\}$, i.e., a hypercube with side lengths $\bar{\epsilon}$. Unless stated otherwise, we use the tolerance $\epsilon_p = 10^{-6}$.

\begin{remark}
The reported experiments are by no means exhaustive of the possible analyzes that the proposed lift-partition-project scheme enables. To get more intricate results, additional structure in the solver, for example, how matrix updates are performed, needs to be specified (which is, again, deliberately abstracted away in this paper).
\end{remark}

We apply the certification method in \cite{arnstrom2022unifying} extended with the lift-partition-project scheme in Algorithm~\ref{alg:lpp} on a multi-parametric quadratic program of the form \eqref{eq:mpqp} that originates from the MPC of  
an inverted pendulum on a cart, which is a tutorial problem in the Model Predictive Control Toolbox in MATLAB. Specifically, the resulting mpQP has the dimensions $n_x = 5, m=10$, and $n_{\theta} = 8$.

\pgfplotstableread{data/invpend_slack.dat}{\invpendslack}
\pgfplotstableread{data/iter_accum.dat}{\iterrate}

First, we use the proposed framework to analyze how the worst-case primal slack $\mu$, taken over all regions, changes as the number of iterations increases, for different error upper bounds $\bar{\epsilon}$. 
The result is shown in Figure \ref{fig:wc-slack}. 
(Recall that dual active-set algorithms terminate if  $\mu \geq -\epsilon_p$ or, equivalently, if $-\mu \leq \epsilon_p$.) 
\begin{figure}
  \centering
  \begin{tikzpicture}[scale=0.8]
	\begin{axis}[
	  xmin=1,xmax=14,
	  ymin=0.00001,
	  ymode=log,
	  xlabel={\# of iterations},
	  minor y tick num=1,
	  ylabel={$\underset{\theta}{\max} -\mu(\theta)$},
	  legend style={at ={(1,1)},anchor=north east}, ymajorgrids,yminorgrids,xmajorgrids,
	  x post scale=1.25,
	  y post scale=0.7,
	  legend cell align={left},legend columns=1,
	  ]
	  \addplot [set19c3,very thick] table [x={iter}, y={eps1e-3}] {\invpendslack};
	  \addplot [set19c2,very thick] table [x={iter}, y={eps1e-4}] {\invpendslack};
	  \addplot [set19c1,very thick] table [x={iter}, y={eps0}] {\invpendslack};
	  \legend{
		$\bar{\epsilon} = 1\cdot 10^{-3}$, 
		$\bar{\epsilon} = 1\cdot 10^{-4}$, 
		$\bar{\epsilon} = 0$
	  };
	\end{axis}
  \end{tikzpicture}
  \caption{Worst-case primal slack for any $\theta \in \Theta_0$ after varying number of executed iterations and error bounds $\bar{\epsilon}$.}%
  \label{fig:wc-slack}
\end{figure}
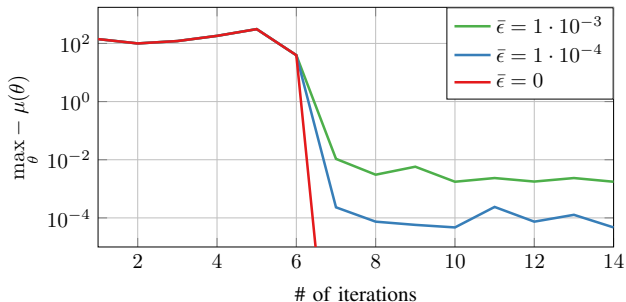

In perfect arithmetic ($\bar{\epsilon}=0$) the worst-case primal slack goes below the tolerance $\epsilon_p=10^{-6}$ after 8 iterations for all parameters of interest. The certification method, hence, concludes that the worst-case number of iterations is 8. If, on the other hand, the upper bound on the errors is $10^{-4}$ or $10^{-3}$ the worst-case slack never becomes lower than $10^{-6}$ before reaching the iteration limit (which was set to 15). Figure \ref{fig:wc-slack} also illustrates that the worst-case slack barely improves after 8 iterations. Hence, if some primal infeasibility above $10^{-6}$ is acceptable, one could use an ad hoc termination criterion of always terminating after 8 iterations, and the proposed framework provide guarantees on how much the primal infeasibility becomes in the worst-case (even after numerical errors have been accounted for.) Without the guarantees provided by the proposed method, such an ad hoc rule might lead to unexpected behaviour, since the primal slack is not necessarily monotonically increasing in dual active-set methods (which can also be seen in Figure \ref{fig:wc-slack} as $-\mu$ is not monotonically decreasing).
	
The worst-case number of iterations for different upper bounds $\bar{\epsilon}$ on the error and different tolerances $\epsilon_p$ is reported in Table \ref{tab:tol-vs-error},
which highlights the (intuitive) necessity of selecting tolerances that are \textit{error-estimating}, i.e., larger than the magnitudes of errors, to ensure that active-set methods are well-behaved (cf. Definition 5.2 and Theorem 5.3 in \cite{ogryczak1988simplex} for details). The proposed method, hence, enables us to analyze the exact interactions of numerical errors and tolerances for a given linear MPC problem. 
\begin{table}
  \centering
  \caption{Worst-case number of iterations for varying values on the error bound $\bar{\epsilon}$ and the primal feasibility tolerance $\epsilon_p$. An entry with $\infty$ means that the iteration limit was reached.}
  \label{tab:tol-vs-error}
  \begin{tabular}{c | c c c }
	\hline \hline
	$\epsilon_p \backslash \bar{\epsilon}$ & 0 & $10^{-4}$ & $10^{-3}$   \\  \hline 
	$10^{-6}$  & 8 & $\infty$ & $\infty$  \\ 
	$10^{-4}$  & 8& 9 & $\infty$  \\ 
	$10^{-3}$  & 7 & 8 & 11  \\ 
	\hline \hline
  \end{tabular}
\end{table}

To highlight that the analysis can be used for more than just analyzing worst-case behaviour of the solver, Figure \ref{fig:ratio} shows the percentage of regions that terminate before a certain number of iterations have been performed for different upper bounds $\bar{\epsilon}$ on the error. Figure \ref{fig:ratio} illustrates that not only the worst-case number of iterations increases with larger errors, but also the median number of iterations increases. Note that proposed method can identify $\textit{exactly}$ which parameters $\theta$ result in the algorithm terminating after a given number of iterations.    
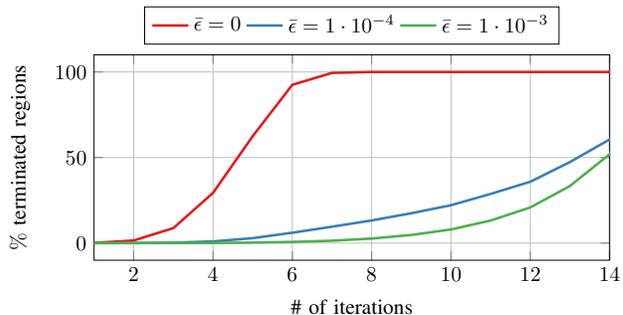
\begin{figure}
  \centering
  \begin{tikzpicture}[scale=0.8]
	\begin{axis}[
	  xmin=1,xmax=14,
	  %ymin=0.0001,
	  %ymode=log,
	  xlabel={\# of iterations},
	  %minor y tick num=1,
	  ylabel={$\%$ terminated regions},
	  legend style={at ={(0.5,1.05)},anchor=south}, ymajorgrids,yminorgrids,xmajorgrids,
	  x post scale=1.25,
	  y post scale=0.6,
	  y filter/.code={\pgfmathparse{#1*100}\pgfmathresult},
	  legend cell align={left},legend columns=3,
	  ]
	  \addplot [set19c1,very thick] table [x={iter}, y={eps0}] {\iterrate};
	  \addplot [set19c2,very thick] table [x={iter}, y={eps1e-4}] {\iterrate};
	  \addplot [set19c3,very thick] table [x={iter}, y={eps1e-3}] {\iterrate};
	  \legend{
		$\bar{\epsilon} = 0$,
		$\bar{\epsilon} = 1\cdot 10^{-4}$, 
		$\bar{\epsilon} = 1\cdot 10^{-3}$
	  };
	\end{axis}
  \end{tikzpicture}
  \caption{Percentage of parameter regions corresponding to the solver terminating after a certain number of iterations.}%
  \label{fig:ratio}
\end{figure}

\section{Conclusion}
We have proposed a general framework that extends parametric complexity certification methods to account for numerical errors that might occur internally in the solvers that are certified. 
Numerical errors in an iteration are accounted for in three steps. First we extend the parameter space, which normally consists of system states and setpoints, with parameters representing the numerical errors. We then partition the extended parameter space based on an iteration in the solver to be certified, similarly to what is normally done in these certification methods in the nominal parameter space. Finally, to retain tractability, we project the resulting regions from the extended parameter space back onto the nominal parameter space. 

To illustrate possible analyses that the extension enables, experiments where the extension was incorporated in the complexity certification of a dual active-set solver were performed. These experiments highlight, for example, that the proposed lift-partition-project scheme can be used to rigorously analyze the interconnection between solver tolerances and numerical errors, which can be used to tune tolerances \textit{a priori}.

More generally, by allowing a rigorous analysis of how numerical errors affect the behaviour of active-set QP solvers, the proposed lift-partition-project scheme improves the reliability of applying such solvers in safety-critical MPC applications. 

\bibliographystyle{IEEEtran}
\linespread{1.0}\selectfont
\bibliography{lib.bib}
\end{document}